\documentclass[12pt,a4paper,reqno,twoside]{amsart}

\usepackage[english]{babel}
\usepackage{stmaryrd}
\usepackage{dsfont}
\usepackage[symbol*,ragged]{footmisc}
\usepackage[colorlinks,linkcolor=red,anchorcolor=blue,citecolor=blue,urlcolor=blue]{hyperref}
\usepackage{color,xcolor}

\usepackage{geometry}
\usepackage{amssymb}
\usepackage{amsmath}
\usepackage{mathrsfs}
\usepackage{amsfonts}
\usepackage{epsfig}

\usepackage{amsthm}
\usepackage{amsxtra}
\usepackage{bbding}
\usepackage{epsfig}
\usepackage{graphicx}
\usepackage{latexsym}
\usepackage{mathbbol}
\usepackage{bbold}

\usepackage{pifont}
\usepackage{wasysym}
\usepackage{skull}
\usepackage{float}

\DeclareSymbolFontAlphabet{\mathbb}{AMSb}
\DeclareSymbolFontAlphabet{\mathbbol}{bbold}

\usepackage{amscd}
\usepackage[all]{xy}
\allowdisplaybreaks[4]
\usepackage{setspace}

\theoremstyle{plain}

\newtheorem*{theorem*}{\normalfont\scshape Theorem}
\newtheorem{proposition}{\normalfont\scshape Proposition}[section]
\newtheorem{lemma}[proposition]{\normalfont\scshape Lemma}

\newtheorem*{corollary*}{\normalfont\scshape Corollary}

\theoremstyle{remark}
\newtheorem*{remark*}{\normalfont\scshape Remark}

\numberwithin{equation}{section}
\addtocounter{footnote}{1}

\renewcommand{\footnoterule}{
  \kern -3pt
  \hrule width 2.5in height 0.4pt
  \kern 3pt
}

\makeatletter
\@ifundefined{MakeUppercase}{}{}
\makeatother

\begin{document}
	
\title[ On high power moments of the error term of the Dirichlet divisor function over primes ]
	  { On high power moments of the error term of the Dirichlet divisor function over primes }

\author[Zhen Guo, Xin Li]{Zhen Guo \quad \& \quad Xin Li}

\address{Department of Mathematics, China University of Mining and Technology,
         Beijing, 100083, People's Republic of China}

\email{zhen.guo.math@gmail.com}


\address{Department of Mathematics, China University of Mining and Technology,
         Beijing, 100083, People's Republic of China}

\email{lixin\_alice@foxmail.com}

\date{}

\footnotetext[1]{Zhen Guo is the corresponding author. 

\quad\,\,
{\textbf{MR(2020) Subject Classification}}: 11L07, 11L20, 11N37.

}

\begin{abstract}
Let $3\leqslant k\leqslant9$ be a fixed integer, $p$ be a prime and $d(n)$ denote the Dirichlet divisor function. We use $\Delta(x)$ to denote the error term in the asymptotic formula of the summatory function of $d(n)$. The aim of this paper is to study the $k$-th power moments of $\Delta(p)$, namely $\sum_{p\leqslant x}\Delta^k(p)$, and we give an asymptotic formula.
\end{abstract}

\maketitle

\section{Introduction and Main Result}
Let $d(n)$ denote the Dirichlet divisor function. The general divisor problem is to study the asymptotic behavior of the sum
\begin{equation*}
  D(x):=\sum_{n\leqslant x}d(n)
\end{equation*}
as $x$ tends to infinity. It is well known that
\begin{equation*}
  D(x)=x\log x+(2\gamma-1)x+\Delta(x),
\end{equation*}
where $\gamma$ is the Euler's constant and $\Delta(x)$ is called the ``error term''. It was first proved by Dirichlet that $\Delta(x)=O(x^{1/2})$. And the exponent $1/2$ was improved by many authors\cite{MR1199067,MR0644943,MR1512430,MR1509041}. Until now the best result
\begin{equation}\label{huxley}
  \Delta(x)\ll x^{131/416}(\log x)^{26497/8320}
\end{equation}
was given by Huxley\cite{MR2005876}.

It is conjectured that
\begin{equation*}
  \Delta(x)=O(x^{1/4+\varepsilon})
\end{equation*}
holds for any $\varepsilon>0$, which is supported by the classical mean-square result 
\begin{equation}\label{lau and tsang}
  \int_{1}^{T}\Delta^2(x)dx=C_2T^{3/2}+O(T\log^3T\log\log T)
\end{equation}
proved by Lau and Tsang\cite{MR2475967}, where $C_2$ is a positive constant which can be written down explicitly by the Riemann zeta-function $\zeta(s)$, and the upper bound estimate
\begin{equation}\label{ivic}
  \int_{1}^{T}|\Delta(x)|^{A_0}dx\ll T^{1+\frac{A_0}{4}+\varepsilon},
\end{equation}
where $A_0>2$ is a real number. The estimate of type (\ref{ivic}) can be found in Ivi\'{c}\cite{MR0792089} with $A_0=35/4$. In this paper we use Huxley's result (\ref{huxley}) and Ivi\'{c}'s\cite{MR0792089} method to obtain that we can take $A_0=262/27$(similar to Lemma \ref{Delta A power moment} in this paper). 

For any integer $3\leqslant k\leqslant9$, we have the asymptotic formula
\begin{equation}\label{Delta k power moment}
  \int_{1}^{T}\Delta^k(x)dx=C_kT^{1+k/4}+O(T^{1+\frac{k}{4}-\delta_k+\varepsilon}),
\end{equation}
where $C_k$ and $\delta_k>0$ are explicit constants. Tsang\cite{MR1162488} first proved the asymptotic formula (\ref{Delta k power moment}) for $k=3$ and $k=4$ with $\delta_3=1/14$ and $\delta_4=1/23$. Ivi\'{c} and Sargos\cite{MR2342663} proved the asymptotic formula (\ref{Delta k power moment}) for $k=3$ and $k=4$ with $\delta_3=7/20$ and $\delta_4=1/12$. Zhai\cite{MR2067871} proved (\ref{Delta k power moment}) for any integer $3\leqslant k\leqslant9$ by a unified approach. Especially for $k=5,6,7,8,9$, he showed that
\begin{equation*}
  \delta_5=\frac{1}{64}, \qquad\delta_6=\frac{35}{4742}, \qquad\delta_7=\frac{17}{6312},\qquad\delta_8=\frac{8}{9433},\qquad\delta_9=\frac{13}{75216}.
\end{equation*}
Zhai\cite{MR2168766} proved $\delta_4=3/28$. Zhang and Zhai\cite{MR2776009} proved that $\delta_5=3/80$. Li\cite{MR3606942} proved that $\delta_7=1/336$. Cao, Tanigawa and Zhai\cite{MR4421948} improved the error term in (\ref{Delta k power moment}) for $k\in\{6,7,8,9\}$.

Also, the discrete mean values
\begin{equation*}
  \mathcal{D}_k(x):=\sum_{n\leqslant x}\Delta^k(n),\qquad 1\leqslant k,\qquad k \in\mathbb{N}
\end{equation*}
have been investigated by various authors. Define the continuous mean values
\begin{equation*}
  \mathcal{C}_k(x):=\int_{1}^{x}\Delta^k(t)dt.
\end{equation*}
Many authors found that the discrete and the continuous mean value formulas are connected deeply with each other and studied the difference between $\mathcal{D}_k(x)$ and $\mathcal{C}_k(x)$. The both results $k=1$ were obtained by Voronoi\cite{MR1580627,MR1509041}. Hardy\cite{MR1576556} studied the difference between $\mathcal{D}_2(x)$ and $\mathcal{C}_2(x)$ and derived that $\mathcal{D}_2(x)=O(x^{3/2+\varepsilon})$. These differences in general power moments were studied by Furuya \cite{MR2176479}, he gave the asymptotic formulas
\begin{equation*}
  \mathcal{D}_2(x)=\mathcal{C}_2(x)+\frac{1}{6}x\log^2x+\frac{8\gamma-1}{12}x\log x+\frac{8\gamma^2-2\gamma+1}{12}x+\left\{
  \begin{array}{lr}
    O \\
    \Omega_\pm
  \end{array}
  \right\}(x^{3/4}\log x),
\end{equation*}
and
\begin{equation*}
  \mathcal{D}_3(x)=\mathcal{C}_3(x)+\frac{3}{2}C_2x^{3/2}\log x+(3\gamma-1)C_2x^{3/2}+\left\{
  \begin{array}{lr}
    O(x\log^5x), \\
    \Omega_{-}(x\log^3x),
  \end{array}
  \right.
\end{equation*}
where $C_2$ is defined in (\ref{lau and tsang}), and higher power cases
\begin{equation*}
  \mathcal{D}_k(x)=\mathcal{C}_k(x)+\left\{
  \begin{array}{lr}
    O(x^{(k+3)/4+\varepsilon}) ,& 4\leqslant k\leqslant10, \\
    O(x^{(35k+3)/108+\varepsilon}),& k\geqslant11,
  \end{array}
  \right.
\end{equation*}
which connect $\mathcal{D}_k(x)$ and $\mathcal{C}_k(x)$ for $k\geqslant2$. Cao, Furuya, Tanigawa and Zhai\cite{MR3231408} improved the results for $4\leqslant k\leqslant 10$.

As analogues of the discrete mean values, we can consider the sum of $\Delta(\cdot)$ over a subset of $\mathbb{N}$. For example, the sum over primes, namely
\begin{equation}\label{Delta k p exp}
  \sum_{p\leqslant x}\Delta^k(p)
\end{equation}
for integer $2\leqslant k\leqslant9$. The authors\cite{citekeyUnpublished} studied the case $k=2$. The aim of this paper is to prove that there exists asymptotic formulas of the sum (\ref{Delta k p exp}) for $3\leqslant k\leqslant9$. Our result is as follows.
\begin{theorem*}
  For any integer $3\leqslant k\leqslant9$, we have
\begin{equation*}
  \sum_{p\leqslant x}\Delta^k(p)=B_k(\infty)\sum_{p\leqslant x}p^{k/4}+O(x^{1+\frac{k}{4}-\theta(k,A_0)+\varepsilon}),
\end{equation*}
where $B_k(\infty)(3\leqslant k\leqslant9)$ are computable constants, $p$ runs over primes and
\begin{equation*}
  \theta(k,A_0)=\left\{
  \begin{array}{lr}
    \min\left(\frac{1}{16k},\frac{1}{2^{k+1}+4k+4}\right),&2<k\leqslant8, \\
    \left(4+265\frac{A_0-4}{A_0-k}\right)^{-1},& k=9,
  \end{array}
  \right.
\end{equation*}
here $A_0$ is defined in (\ref{ivic}) and we take $A_0=262/27$. More precisely, we have
\begin{equation*}
\begin{aligned}
  \theta(3,262/27)&=\frac{1}{48}, \qquad\theta(4,262/27)=\frac{1}{64}, \qquad\theta(5,262/27)=\frac{1}{88},\qquad\theta(6,262/27)=\frac{1}{156},\\
  \theta(7,262/27)&=\frac{1}{288},\qquad\theta(8,262/27)=\frac{1}{548},\qquad\theta(9,262/27)=\frac{19}{40848}.
\end{aligned}
\end{equation*}
\end{theorem*}
\section{Priliminary}

\subsection{Notations}
Throught this paper, $\varepsilon$ denotes a sufficiently small positive number, not necessarily the same at each occurrence. Let $p$ always denote a prime number. As usual, $d(n)$, $\Lambda(n)$ and $\mu(n)$ denote the Dirichlet divisor function, the von Mangoldt function and the M\"{o}bius function, respectively, $\mathbb{C}$ denotes the set of complex numbers, $\mathbb{R}$ denotes the set of real numbers, and $\mathbb{N}$ denotes the set of natural numbers. We write $\mathbf{e}(t):=exp(2\pi it)$ . The notation  $m\sim M$ means $M<m\leqslant2M$, $f(x)\ll g(x)$ means $f(x)=O(g(x))$; $f(x)\asymp g(x)$ means $f(x)\ll g(x)\ll f(x)$. $\gcd(m_1,\cdots,m_j)$ denotes the greatest common divisor of $m_1,\cdots,m_j$ $\in\mathbb{N}$($j=2,3,\cdots$).

\subsection{Auxiliary Lemmas}
\
\newline \indent 

\begin{lemma}\label{Double large sieve}
  Let $\mathscr{X}:=\{x_r\}$, $\mathscr{Y}:=\{y_s\}$ be two finite sequence of real numbers with 
\begin{equation*}
  |x_r|\leqslant X, \qquad|y_s|\leqslant Y,
\end{equation*}
and $\varphi_r,\psi_s\in\mathbb{C}$. Defining the bilinear forms
\begin{equation*}
  {\mathscr{B}}_{\varphi\psi}(\mathscr{X},\mathscr{Y})=\sum_{r\sim R}\sum_{s\sim S}\varphi_r\psi_s\mathbf{e}(x_ry_s),
\end{equation*}
we have
\begin{equation*}
  |\mathscr{B}_{\varphi\psi}(\mathscr{X},\mathscr{Y})|^2\leqslant20(1+XY)\mathscr{B}_{\varphi}(\mathscr{X},Y)\mathscr{B}_{\psi}(\mathscr{Y},X),
\end{equation*}\
with
\begin{equation*}
  \mathscr{B}_{\varphi}(\mathscr{X},Y)=\sum_{|x_{r_1}-x_{r_2}|\leqslant Y^{-1}}|\varphi_{r_1}\varphi_{r_2}|
\end{equation*}
and $\mathscr{B}_{\psi}(\mathscr{Y},X)$ defined similarly.
\end{lemma}
\begin{proof}
  See Proposition 1 of Fouvry and Iwaniec\cite{MR1027058}.
\end{proof}

\begin{lemma}\label{H-B identity}
  Let $z\geqslant1$ be a real number and $k\geq1$ be an integer. Then for any $n\leqslant 2z^k$, there holds
\begin{equation}
  \Lambda(n)=\sum_{j=1}^{k}(-1)^{j-1}\binom{k}{j}\mathop{\sum\cdots\sum}_{\substack{n_1n_2\cdots n_{2j}=n\\n_{j+1},\cdots,n_{2j}\leqslant z}}(\log n_1)\mu(n_{j+1})\cdots\mu(n_{2j}).
\end{equation}
\end{lemma}
\begin{proof}
 See the argument on pp.1366-1367 of Heath-Brown\cite{MR0678676}.
\end{proof}

\begin{lemma}\label{(2.2) of Zhai}
  Let $\{a_i\}$ and $\{b_i\}$ be two finite sequences of real numbers and let $\delta>0$ be a real number. Then we have
\begin{equation*}
  \#\{(r,s):|a_r-b_s|\leqslant\delta\}\leqslant3(\#\{(r,r'):|a_r-a_{r'}|\leqslant\delta\})^{1/2}(\#\{(s,s'):|b_s-b_{s'}|\leqslant\delta\})^{1/2}
\end{equation*}
\end{lemma}
\begin{proof}
  See the argument of P266-267 in Zhai\cite{MR2168766}.
\end{proof}

\begin{lemma}\label{sub alphak}
For an fixed integer $k\geq3$, let $m_1,\cdots,m_k$ are integers, $(i_1,\cdots,i_{k-1})\in\{0,1\}^{k-1}$ such that
\begin{equation*}
  \alpha_k=\sqrt{m_1}+(-1)^{i_1}\sqrt{m_2}+\cdots+(-1)^{i_{k-1}}\sqrt{m_k}\neq0,
\end{equation*} 
then
\begin{equation*}
  |\alpha_k|\gg\max(m_1,\cdots,m_k)^{-(2^{k-2}-\frac{1}{2})}.
\end{equation*} 
\end{lemma}
\begin{proof}
  See Lemma 2.2 of Zhai\cite{MR2067871}.
\end{proof}

\begin{lemma}\label{N O solutions}
  Let $k\geqslant3$ be a fixed integer,$(i_1,\cdots,i_{k-1})\in\{0,1\}^{k-1}$, and $N_1,\cdots,N_k>1$ are real numbers, $E=\max(N_1,\cdots,N_k)$, $\rho$ is a real number such that $0<\rho<E^{1/2}$. Let $\mathscr{A}=\mathscr{A}(N_1,\cdots,N_k,i_1,\cdots,i_{k-1},\rho)$ denote the number of solutions of 
\begin{equation*}
  |\sqrt{n_1}+(-1)^{i_1}\sqrt{n_2}+\cdots+(-1)^{i_{k-1}}\sqrt{n_k}|<\rho.
\end{equation*}
Then
\begin{equation*}
  \mathscr{A}\ll\rho E^{-1/2}N_1\cdots N_k+E^{-1}N_1\cdots N_k.
\end{equation*}
\end{lemma}
\begin{proof}
  See Lemma 2.4 of Zhai\cite{MR2067871}.
\end{proof}

\begin{lemma}\label{iwaniec and sarkozy}
  Let $N$ be a given large integer, $X$ be a real number and $\mathbf{N}$ denote the set $\{N+1,\cdots,2N\}$. Suppose 
\begin{equation*}
  v(\mathbf{N},X)=|\{n_1,n_2\in\mathbf{N};|\sqrt{n_1}-\sqrt{n_2}|\leqslant(2X)^{-1}\}|,
\end{equation*}
then we have
\begin{equation*}
  v(\mathbf{N},X)\leqslant(1+2\sqrt{2N}X^{-1})|\mathbf{N}|.
\end{equation*}
\end{lemma}
\begin{proof}
  See Lemma 2 of Iwaniec and S\'{a}rk\"{o}zy\cite{MR0883536}.
\end{proof}

\begin{lemma}
  Suppose $k\geqslant2$ is a fixed integer, $y>1$ is a large parameter. Define
\begin{equation*}
\begin{aligned}
  s_{k,l}(\infty)&:=\sum_{\sqrt{n_1}+\cdots+\sqrt{n_l}=\sqrt{n_{l+1}}+\cdots+\sqrt{n_k}}\frac{d(n_1)\cdots d(n_k)}{(n_1\cdots n_k)^{3/4}}\qquad(1\leqslant l<k),\\
  s_{k,l}(y)&:=\sum_{\substack{\sqrt{n_1}+\cdots+\sqrt{n_l}=\sqrt{n_{l+1}}+\cdots+\sqrt{n_k}\\n_1,\cdots,n_k\leqslant y}}\frac{d(n_1)\cdots d(n_k)}{(n_1\cdots n_k)^{3/4}}\qquad(1\leqslant l<k).
\end{aligned}
\end{equation*}
Then we have

(i)$s_{k,l}(\infty)$ is convergent.

(ii)
\begin{equation*}
  |s_{k,l}(\infty)-s_{k,l}(y)|\ll y^{-1/2+\varepsilon},\qquad1\leqslant l<k.
\end{equation*}
\end{lemma}
\begin{proof}
  See Lemma 3.1 of Zhai\cite{MR2067871}.
\end{proof}

\begin{lemma}\label{deravative est}
  Suppose that $f(x):[a,b]\rightarrow\mathbb{R}$ has continuous derivatives of arbitrary order on $[a,b]$, where $1\leqslant a<b\leqslant2a$. Suppose further that for either $j=1$ or $2$ we have 
\begin{equation*}
  |f^{(j)}(x)|\asymp\lambda_j,\qquad x\in[a,b].
\end{equation*}
Then
\begin{equation}\label{first order derivative}
  \sum_{a<n\leqslant b}\mathbf{e}(f(n))\ll{\lambda_1}^{-1}.
\end{equation}
\begin{equation}\label{second order derivative}
  \sum_{a<n\leqslant b}\mathbf{e}(f(n))\ll a{\lambda_2}^{1/2}+{\lambda_2}^{-1/2}.
\end{equation}
\end{lemma}
\begin{proof}
See Lemma 2.1 of Graham and Kolesnik\cite{MR1145488}.
\end{proof}

\begin{lemma}\label{Large value estimate}
  Let $1\leqslant t_1<t_2<\cdots<t_R\leqslant T$ and $|t_r-t_s|\geqslant V$ for $r\neq s\leqslant R$. If $\Delta(t_r)\gg V> T^{7/32+\varepsilon}$ for $r\leqslant R$, then
\begin{equation*}
  R\ll T^{\varepsilon}(TV^{-3}+T^{15/4}V^{-12}).
\end{equation*}
\end{lemma}
\begin{proof}
See Theorem 13.8 in Ivi\'{c}\cite{MR0792089}.
\end{proof}

\begin{lemma}\label{Heath and Tsang}
Let $T$ be a large real number and $H$ be a real number such that $1\leqslant H\leqslant T$. We have
\begin{equation*}
\begin{aligned}
  \int_{T}^{2T}\max_{h\leqslant H}(\Delta(t+h)-\Delta(t))^2dt\ll HT\log^5T.
\end{aligned}
\end{equation*}
\end{lemma}
\begin{proof}
  See Heath-Brown and Tsang\cite{MR1295953}.
\end{proof}

\begin{lemma}\label{Delta A power moment}
  For any fixed real number $0<A\leqslant262/27=9.70...$ and a large real number $x$, we have the estimate
\begin{equation}\label{1}
  \sum_{x<p\leqslant2x}|\Delta(p)|^A\ll x^{1+A/4+\varepsilon}.
\end{equation}
\end{lemma}
\begin{proof}
  We use the large value estimate argument given by Ivi\'{c}\cite{MR0792089}. We have
\begin{equation*}
  \sum_{x<p\leqslant2x}|\Delta(p)|^A\ll\sum_{x<n\leqslant2x}|\Delta(n)|^A:=\mathrm{U}.
\end{equation*}
Suppose $\Delta(x)\ll x^{\theta+\varepsilon}$. Let $J_1,J_2$ be integers such that $2^{J_1}\asymp x^{1/4}$, $2^{J_2}\asymp x^{\theta}$, then
\begin{equation}\label{U chaifen}
\begin{aligned}
  \mathrm{U}&=\sum_{j}\sum_{2^j<|\Delta(n)|\leqslant 2^{j+1}}|\Delta(n)|^A\\
  &=\sum_{j=1}^{J_1}\sum_{|\Delta(n)|\sim2^j}|\Delta(n)|^{A}+\sum_{j=J_1+1}^{J_2}\sum_{|\Delta(n)|\sim2^j}|\Delta(n)|^{A}\\
  &:=G_1+G_2
\end{aligned}
\end{equation}
We can easily obtain 
\begin{equation}\label{G1 est}
  G_1\ll x^{1+\frac{A}{4}},
\end{equation}
and
\begin{equation*}
  G_2\ll\sum_{j=J_1+1}^{J_2}2^{jA}\sum_{|\Delta(n)|\sim2^j}1.
\end{equation*}
For real $u>n$, we have $|\Delta(u)-\Delta(n)|\ll|u-n|\log u+u^\varepsilon$, thus
\begin{equation*}
  |\Delta(u)|\geqslant|\Delta(n)|-|\Delta(u)-\Delta(n)|\gg 2^j-(u-n)\log u+u^{\varepsilon}.
\end{equation*}
So we can obtain that, in order to make $|\Delta(u)|\gg 2^j$, we need$|u-n|\ll 2^j/\log x$. Then using Lemma \ref{Large value estimate} we deduce that 
\begin{equation}\label{G2 est}
  G_2\ll\sum_{j=J_1+1}^{J_2}2^{j(A+1)}\left(\frac{x^{1+\varepsilon}}{2^{3j}}+\frac{x^{15/4+\varepsilon}}{2^{12j}}\right)\ll x^{1+\frac{A}{4}+\varepsilon}
\end{equation} 
holds for $A\leqslant2\theta/(\theta-1/4)$, here we take the exponent $\theta=131/416$ of Huxley\cite{MR1199067}, then (\ref{1}) holds by (\ref{U chaifen}), (\ref{G1 est}) and (\ref{G2 est}).
\end{proof}

\begin{lemma}\label{delta 1 10 power moment}
For real numbers $N$ and $x$ satisfying $1\ll N\ll x$, we define
\begin{equation*}
  \delta_1(x,N)=\frac{x^{1/4}}{\sqrt{2}\pi}\sum_{n\leqslant N}\frac{d(n)}{n^{3/4}}\cos\left(4\pi\sqrt{nx}-\frac{\pi}{4}\right)
\end{equation*}
Then uniformly for $N\ll x^{1/4}$ we have
\begin{equation*}
  \sum_{x<p\leqslant 2x}|\delta_1(p,N)|^{10}\ll x^{7/2+\varepsilon}.
\end{equation*}
\end{lemma}

\begin{proof}
If $N\ll x^{1/4}$, we the have $\delta_1(p,N)\ll p^{1/4}N^{1/4}\log x\ll x^{5/16}\log x$. So using the same argument in Lemma \ref{Delta A power moment} (i.e. Ivi\'{c}'s large-value technique) on $\delta_1(p,N)$ directly, we can finish the proof of the Lemma. We omit the details.
\end{proof}

\begin{lemma}\label{delta 2 exp}
For real numbers $N$ and $x$ satisfying $1\ll N\ll x$, we define
\begin{equation*}
  \delta_2(x,N):=\Delta(x)-\delta_1(x,N).
\end{equation*}
We have 
\begin{equation*}
  \delta_2(x,N)\ll x^{1/2+\varepsilon}N^{-1/2}.
\end{equation*}

Moreover, if $N\ll x^{1/3}$, then for any fixed real number $2<A\leqslant A_0=262/27$ we have the estimate
\begin{equation}\label{delta 2 A power moment}
  \sum_{x<p\leqslant 2x}|\delta_2(p,N)|^{A}\ll\left\{
  \begin{array}{lr}
    x^{1+\frac{A}{4}+\varepsilon}N^{-\frac{A}{4}},&2<A\leqslant4, \\
    x^{1+\frac{A}{4}+\varepsilon}N^{-\frac{A_0-A}{A_0-4}},& 4<A<A_0.
  \end{array}
  \right.
\end{equation}
\end{lemma}

\begin{proof}
From Lemma \ref{delta 1 10 power moment} and H\"{o}lder's inequality we get the estimate 
\begin{equation*}
  \sum_{x<p\leqslant 2x}|\delta_1(p,N)|^{A}\ll x^{1+A/4+\varepsilon}(0<A\leqslant10),
\end{equation*}
which combning Lemma \ref{Delta A power moment} with $A_0=262/27$ gives
\begin{equation}\label{delta 2 A0 power moment}
\begin{aligned}
  \sum_{x<p\leqslant 2x}|\delta_2(p,N)|^{A_0}&=\sum_{x<p\leqslant 2x}|\Delta(p)-\delta_1(p,N)|^{A_0}\\
  &\ll\sum_{x<p\leqslant 2x}|\Delta(p)|^{A_0}+\sum_{x<p\leqslant 2x}|\delta_1(p,N)|^{A_0}\ll x^{1+A_0/4+\varepsilon}.
\end{aligned}
\end{equation}

We also have
\begin{equation}\label{delta 2 mean square}
  \sum_{x<p\leqslant 2x}{\delta_2}(p,N)^2\ll x^{3/2+\varepsilon}N^{-1/2}.
\end{equation}
In fact, one has
\begin{equation*}
\begin{aligned}
  \sum_{x<p\leqslant 2x}{\delta_2}(p,N)^2&=\sum_{x<p\leqslant2x}\int_{p-1}^{p}\delta_2(p,N)^2dt\\
  &=\sum_{x<p\leqslant2x}\left[\int_{p-1}^{p}\delta_2(t,N)^2dt-\int_{p-1}^{p}(\delta_2(t,N)^2-\delta_2(p,N)^2)dt\right]\\
  &\leqslant\int_{x}^{2x}\delta_2(t,N)^2dt+\sum_{x<p\leqslant2x}\int_{p-1}^{p}(\delta_2(t,N)^2-\delta_2(p,N)^2)dt\\
  &\ll\frac{x^{3/2}\log^3x}{{N}^{1/2}}+x\log^4x+\sum_{x<p\leqslant2x}\int_{p-1}^{p}(\delta_2(t,N)^2-\delta_2(p,N)^2)dt.
\end{aligned}
\end{equation*}
The latter term can be transformed into $E_1-E_2$, where
\begin{equation*}
\begin{aligned}
  E_1&=\sum_{x<p\leqslant2x}\int_{p-1}^{p}(\delta_2(t,N)+\delta_2(p,N))(\Delta(t)-\Delta(p))dt,\\
  E_2&=\sum_{x<p\leqslant2x}\int_{p-1}^{p}(\delta_2(t,N)+\delta_2(p,N))(\delta_1(t,N)-\delta_1(p,N))dt.
\end{aligned}
\end{equation*}
For $E_2$, by the expression of $\delta_1$ and Lagrange's mean value theorem we obtain $\delta_1(t,N)-\delta_1(p,N)\ll t^{-1/4}{N}^{3/4}$, and we have $\delta_2(t,N)+\delta_2(p,N)\ll t^{1/2+\varepsilon}{N}^{-1/2}$, thus 
\begin{equation*}
  E_2\ll\int_{x}^{2x}t^{1/4+\varepsilon}{N}^{1/4}dt\ll x^{5/4+\varepsilon}{N}^{1/4}.
\end{equation*}
For $E_1$, one has
\begin{equation*}
\begin{aligned}
  E_1&\ll\int_{x}^{2x}t^{1/2}{N}^{-1/2}\max_{0<v\leqslant1}|\Delta(t)-\Delta(t+v)|dt\\
  &\ll \frac{x^{1/2+\varepsilon}}{{N}^{1/2}}\int_{x}^{2x}\max_{0<v\leqslant1}|\Delta(t)-\Delta(t+v)|dt\\
  &\ll x^{3/2+\varepsilon}{N}^{-1/2},
\end{aligned}
\end{equation*}
where we use Lemma \ref{Heath and Tsang}. Thus we conclude that 
\begin{equation}\label{S_22 est}
  \sum_{x<p\leqslant 2x}{\delta_2}^2(p,N)\ll x^{3/2+\varepsilon}N^{-1/2}+x^{5/4+\varepsilon}{N}^{1/4}.
\end{equation}
Noting that $N\ll x^{1/3}$, we obtain (\ref{delta 2 mean square}).

Write
\begin{equation*}
  \delta_2(p,N)=\frac{p^{1/4}}{\sqrt{2}\pi}\sum_{N<n\leqslant\sqrt{x}}\frac{d(n)}{n^{3/4}}\cos\left(4\pi\sqrt{np}-\frac{\pi}{4}\right)+\delta_2(p,\sqrt{x}).
\end{equation*}
We have
\begin{equation}\label{delta 2 4 moment chaifen}
  \sum_{x<p\leqslant 2x}{\delta_2}^4(p,N)\ll x\sum_{x<p\leqslant 2x}\left|\sum_{N<n\leqslant\sqrt{x}}\frac{d(n)}{n^{3/4}}\cos\left(4\pi\sqrt{np}-\frac{\pi}{4}\right)\right|^4+\sum_{x<p\leqslant 2x}\left|\delta_2(p,\sqrt{x})\right|^4
\end{equation}

By a splitting argument we get for some $N\ll M\ll \sqrt{x}$ that
\begin{equation*}
\begin{aligned}
  &x\sum_{x<p\leqslant 2x}\left|\sum_{N<n\leqslant\sqrt{x}}\frac{d(n)}{n^{3/4}}\cos\left(4\pi\sqrt{np}-\frac{\pi}{4}\right)\right|^4\\
  &\ll x\log^4x \times\sum_{m_1,m_2,m_3,m_4\sim M}\frac{d(m_1)d(m_2)d(m_3)d(m_4)}{(m_1m_2m_3m_4)^{3/4}}\left|\sum_{x<n\leqslant 2x}\mathbf{e}(2\eta\sqrt{n})\right|\\
  &\ll x^{3/2}\log^4x \times\sum_{m_1,m_2,m_3,m_4\sim M}\frac{d(m_1)d(m_2)d(m_3)d(m_4)}{(m_1m_2m_3m_4)^{3/4}}\min\left(\sqrt{x},\frac{1}{|\eta|}\right),
\end{aligned}
\end{equation*}
where $\eta={n_1}^{1/2}+{n_2}^{1/2}-{n_3}^{1/2}-{n_4}^{1/2}$, and we use (\ref{first order derivative}) of Lemma \ref{deravative est}. Similar to the proof of Lemma 2.15 of Cao, Tanigawa and Zhai\cite{MR4421948}, we get
\begin{equation}\label{delta 2 4 moment p1}
  x\sum_{x<p\leqslant 2x}\left|\sum_{N<n\leqslant\sqrt{x}}\frac{d(n)}{n^{3/4}}\cos\left(4\pi\sqrt{np}-\frac{\pi}{4}\right)\right|^4\ll\frac{x^{2+\varepsilon}}{N}+x^{3/2+\varepsilon}N^{1/2}
  \ll\frac{x^{2+\varepsilon}}{N}
\end{equation}
by noting that $N\ll M\ll\sqrt{x}$ and $N\ll x^{1/8}$.

By the Voronoi's formula of $\Delta(p)$ we get $\delta_2(p,\sqrt{x})\ll x^{1/4+\varepsilon}$, which combing (\ref{delta 2 mean square}) gives
\begin{equation}\label{delta 2 4 moment p2}
  \sum_{x<p\leqslant 2x}\left|\delta_2(p,\sqrt{x})\right|^4\ll x^{1/2+\varepsilon}\sum_{x<p\leqslant 2x}\left|\delta_2(p,\sqrt{x})\right|^2\ll x^{7/4+\varepsilon}\ll x^{2+\varepsilon}N^{-1}.
\end{equation}

From (\ref{delta 2 4 moment chaifen}), (\ref{delta 2 4 moment p1}) and (\ref{delta 2 4 moment p2}) we get the estimate
\begin{equation}\label{delta2 4 moment}
  \sum_{x<p\leqslant 2x}{\delta_2}^4(p,N)\ll x^{2+\varepsilon}N^{-1}\:(N\ll x^{1/8}).
\end{equation}

From (\ref{delta 2 mean square}), (\ref{delta2 4 moment}) and H\"{o}lder's inequality we get the estimate (\ref{delta 2 A power moment}) for the case $2<A\leqslant4$. From (\ref{delta 2 A0 power moment}), (\ref{delta2 4 moment}) and H\"{o}lder's inequality we get the estimate (\ref{delta 2 A power moment}) for the case $4<A\leqslant A_0$.
\end{proof}

For fixed integer $3\leqslant k\leqslant9$, we state a multiple exponential sums of the following form:
\begin{equation}\label{exp sum wlog}
  \mathop{\sum_{m_1\sim M_1}\cdots\sum_{m_k\sim M_k}}_{\substack{\alpha_k\neq0}}d(m_1)\cdots d(m_k)\mathop{\sum_{r\sim R}\sum_{s\sim S}}_{\substack{RS\asymp x}}\xi_r\eta_{s}\mathbf{e}\left(2\alpha_kr^{1/2}{s}^{1/2}\right),
\end{equation}
with $\xi_e, \eta_s\in\mathbb{C}, |\xi_r|\ll x^{\varepsilon}, |\eta_s|\ll x^{\varepsilon}$ and $M_1,\cdots,M_k,R,S,x$ are given real numbers, such that $1\ll M_1,\cdots,M_k\ll x$, and $\alpha_k$ is given in Lemma \ref{sub alphak}. In this sum, without loss of generality we assume $m_1=\max(m_1,\cdots,m_k)$. It is called a "Type I" sum, denoted by $S_I(k)$, if $\eta_s=1$ or $\eta_s=\log s$; otherwise it is called a "Type II" sum, denoted by $S_{II}(k)$.

\begin{lemma}\label{Type I est}
  Suppose that $\xi_r\ll1$, $\eta_s=1$ or $\eta_s=\log\ell$, $RS\asymp x$. Then for $R\ll x^{1/4}$, there holds
\begin{equation*} 
  x^{-\varepsilon}S_I(k)\ll x^{1/2}{M_1}^{5/4}(M_2\cdots M_k)+x^{3/4}{M_1}^{(2^{k-3}+3/4)}(M_2\cdots M_k).
\end{equation*}
\end{lemma}

\begin{proof}
  Set $f(s)=2\alpha_k\sqrt{rs}$. It is easy to see that 
\begin{equation*}
  f''(s)=-\frac{1}{2}\alpha_kr^{1/2}s^{-3/2}.
\end{equation*}
If $R\ll x^{1/4}$, then by Lemma \ref{deravative est}, we deduce that
\begin{equation*}
\begin{aligned}
  x^{-\varepsilon}S_I(k)&\ll\mathop{\sum_{m_1\sim M_1}\cdots\sum_{m_k\sim M_k}}_{\substack{\alpha_k\neq0}}\sum_{r\sim R}\left|\sum_{s\sim S}\mathbf{e}(f(s))\right|\\
  &\ll\mathop{\sum_{m_1\sim M_1}\cdots\sum_{m_k\sim M_k}}_{\substack{\alpha_k\neq0}}\sum_{r\sim R}\left(S(|\alpha_k|R^{1/2}S^{-3/2})^{1/2}+(|\alpha_k|R^{1/2}S^{-3/2})^{-1/2}\right)\\
  &\ll x^{1/4}R{M_1}^{5/4}(M_2\cdots M_k)+x^{3/4}\mathop{\sum_{m_1\sim M_1},\cdots,\sum_{m_k\sim M_k}}_{\substack{\alpha_k\neq0}}\frac{1}{|\alpha_k|^{1/2}},
\end{aligned} 
\end{equation*}
 then by Lemma \ref{sub alphak} we have
\begin{equation*}
\begin{aligned}
  x^{-\varepsilon}S_I(k)&\ll x^{1/4}R{M_1}^{5/4}(M_2\cdots M_k)+x^{3/4}{M_1}^{(2^{k-3}+3/4)}(M_2\cdots M_k)\\
  &\ll x^{1/2}{M_1}^{5/4}(M_2\cdots M_k)+x^{3/4}{M_1}^{(2^{k-3}+3/4)}(M_2\cdots M_k).
\end{aligned}
\end{equation*}
\end{proof}

In order to separate the dependence from the range of summation we appeal to the following formula
\begin{lemma}\label{separate}
  Let $0<V\leqslant U<\nu U<\lambda V$ and let $a_{v}$ be complex numbers with $|a_{v}|\leqslant1$. We then have
\begin{equation*}
  \sum_{U<u<\nu U}a_u=\frac{1}{2\pi}\int_{-V}^{V}\left(\sum_{V<v<\lambda V}a_vv^{-it}\right)U^{it}(\nu^{it}-1)t^{-1}dt+O(\log(2+V)),
\end{equation*}
where the constant implied in $O$ depends on $\lambda$ only.
\end{lemma}
\begin{proof}
  See Lemma 6 of Fouvry and Iwaniec\cite{MR1027058}.
\end{proof}
\begin{lemma}\label{Type II est}
  Suppose $|\xi_r|\ll x^{\varepsilon}, |\eta_s|\ll x^{\varepsilon}$ with $r\sim R$, $s\sim S$, $RS\asymp x$. Then for $x^{1/4}\ll R\ll{x^{1/2}}$, there holds
\begin{equation*}
  S_{II}(k)\ll x^{7/8}(M_1\cdots M_k)+x^{15/16}{M_1}^{3/4}(M_2\cdots M_k)+x^{3/4}{M_1}^{(2^{k-3}+1)}(M_2\cdots M_k).
\end{equation*}
\end{lemma}
\begin{proof}
  Applying Lemma \ref{separate} we obtain
\begin{equation}\label{exp SII}
  S_{II}(k)\ll|{S_{II}}^{*}(k)|+R(M_1\cdots M_k)\log x,
\end{equation}
where 
\begin{equation*}
  {S_{II}}^{*}(k)=\mathop{\sum_{m_1\sim M_1}\cdots\sum_{m_k\sim M_k}}_{\substack{\alpha_k\neq0}}d(m_1)\cdots d(m_k)\mathop{\sum_{r\sim R}\sum_{s\sim S}}{\xi_r}^{*}{\eta_{s}}^{*}\mathbf{e}\left(2\alpha_kr^{1/2}{s}^{1/2}\right)
\end{equation*}
with ${\xi_r}^{*},{\eta_s}^{*}\ll x^{\varepsilon}$.

Applying Lemma \ref{Double large sieve} with $\mathscr{X}=2\alpha_kr^{1/2}$, $\mathscr{Y}=s^{1/2}$, we obtain
\begin{equation}\label{T * est1}
  {S_{II}}^{*}\ll x^{1/4}{M_1}^{1/4}(S_{a}S_{b})^{1/2},
\end{equation}
where 
\begin{equation*}
\begin{aligned}
  S_{a}&=\sum_{|\alpha_kr^{1/2}-{\alpha_k}'{r'}^{1/2}|\leqslant S^{-1/2}}d(m_1)\cdots d(m_k)d({m_1}')\cdots d({m_k}')\xi_r{\xi_r}',\\
  S_{b}&=\sum_{|s^{1/2}-{s'}^{1/2}|\leqslant(M_1R)^{-1/2}}\eta_s{\eta_s}',
\end{aligned}
\end{equation*}
and ${\alpha_k}'=\sqrt{{m_1}'}+(-1)^{i_1}\sqrt{{m_2}'}+\cdots+(-1)^{i_{k-1}}\sqrt{{m_k}'}$, $(i_1,\cdots, i_{k-1})\in\{0,1\}^{k-1}$.

By (\ref{mathscr B est}) with $\Delta=S^{-1/2}$, we easily obtain
\begin{equation}\label{S a est}
\begin{aligned}
  x^{-\varepsilon}S_{a}&\ll\mathscr{B}(M_1,M_2,R,S^{-1/2})\\
  &\ll{R}^{3/2}{M_1}^{3/2}(M_2\cdots M_k)^2+x^{-1/2}R^{2}{M_1}^{(2^{k-2}+3/2)}(M_2\cdots M_k)^2.
\end{aligned}
\end{equation}

For $S_b$, using Lemma \ref{iwaniec and sarkozy} we have
\begin{equation}\label{S b est}
\begin{aligned}
  x^{-\varepsilon}S_{b}&\ll\sum_{|s^{1/2}-{s'}^{1/2}|\leqslant(M_1R)^{-1/2}}1\\
  &\ll xR^{-1}+x^{3/2}R^{-2}{M_1}^{-1/2}.  
\end{aligned}
\end{equation}
Thus combining (\ref{exp SII}), (\ref{T * est1}), (\ref{S a est}), (\ref{S b est}) and the condition $x^{1/4}\ll R\ll x^{1/2}$ we conclude that
\begin{equation}\label{T * est2}
\begin{aligned}
  x^{-\varepsilon}S_{II}&\ll x^{1/4}{M_1}^{1/4}(R^{3/4}{M_1}^{3/4}(M_2\cdots M_k)+x^{-1/4}R{M_1}^{(2^{k-3}+3/4)}(M_2\cdots M_k))\\
  &\times(x^{1/2}R^{-1/2}+x^{3/4}R^{-1}{M_1}^{-1/4})+R(M_1\cdots M_k)\\
  &\ll x^{3/4}R^{1/4}(M_1\cdots M_k)+xR^{-1/4}{M_1}^{3/4}(M_2\cdots M_k)+x^{1/2}R^{1/2}{M_1}^{(2^{k-3}+1)}(M_2\cdots M_k)\\
  &+x^{3/4}{M_1}^{(2^{k-3}+3/4)}(M_2\cdots M_k)+R(M_1\cdots M_k)\\
  &\ll x^{7/8}(M_1\cdots M_k)+x^{15/16}{M_1}^{3/4}(M_2\cdots M_k)+x^{3/4}{M_1}^{(2^{k-3}+1)}(M_2\cdots M_k).
\end{aligned}
\end{equation}
\end{proof}

\section{The spacing problem}

Let $k$ be an integer such that $3\leqslant k\leqslant9$, $M_1,M_2,\cdots,M_k\geqslant1$, $R\geqslant1$. In this section $\Delta$ is a small real positive number. We are going to investigate the distributions of real numbers of type
\begin{equation*}
  t(r,m_1,m_2,\cdots,m_k)=\alpha_k r^{\beta}
\end{equation*}
with $0<\beta<1$, $m_j\sim M_j(j=1,\cdots,k)$, $(i_1,\cdots,i_{k-1})\in\{0,1\}^{k-1}$, $\alpha_k={m_1}^{1/2}+(-1)^{i_1}{m_2}^{1/2}+\cdots+(-1)^{i_{k-1}}{m_k}^{1/2}\neq0$, $r\sim R$.

Let $\mathscr{B}(M_1, M_2,\cdots,M_k, R, \Delta)$ denote the number of $(2k+2)$-turplets \\$(m_1, \tilde{m_1}, m_2, \tilde{m_2},\cdots,m_k,\tilde{m_k}, r, \tilde{r})$ with $m_j, \tilde{m_j}\sim M_j(j=1,\cdots,k)$, $r, \tilde{r}\sim R$ satisfying
\begin{equation}\label{mathscr B}
  |t(r,m_1,m_2,\cdots,m_k)-t(\tilde{r},\tilde{m_1},\tilde{m_2},\cdots,\tilde{m_k})|\leqslant\Delta,
\end{equation}
and we denote $t(\tilde{r},\tilde{m_1},\tilde{m_2},\cdots,\tilde{m_k})$ by $\tilde{\alpha_k}{\tilde{r}}^\beta$, where
\begin{equation*}
  \tilde{\alpha_k}={\tilde{m_1}}^{1/2}+(-1)^{i_1}{\tilde{m_2}}^{1/2}+\cdots+(-1)^{i_{k-1}}{\tilde{m_k}}^{1/2}.
\end{equation*}

Our aim is to prove the following:

\begin{proposition}
  We have
\begin{equation}\label{mathscr B est}
\begin{aligned}
  &\mathscr{B}(M_1, M_2,\cdots,M_k, R, \Delta)\\
  &\ll R^{2-\beta}{M_1}^{3/2}(M_2\cdots M_k)^2+\Delta R^{2-\beta}{M_1}^{3/2}(M_2\cdots M_k)^2.
\end{aligned}
\end{equation}
\end{proposition}
\begin{proof}
Suppose $m_1,m_2,\cdots,m_k,\tilde{m_1},\tilde{m_2},\cdots,\tilde{m_k}$ are fixed, without loss of generality we assume that $m_1>m_2>\cdots>m_k$ and $\tilde{m_1}>\tilde{m_2}>\cdots>\tilde{m_k}$. We denote the number of solutions of
\begin{equation}\label{mathscr B1}
  \left|\left(\frac{\alpha_k}{\tilde{\alpha_k}}\right)
  -\left(\frac{r}{\tilde{r}}\right)^{\beta}\right|\leq\frac{\Delta}{|\tilde{\alpha_k}|R^{\beta}}
\end{equation}
by $\mathscr{B}(m_1,m_2,\cdots,m_k,\tilde{m_1},\tilde{m_2},\cdots,\tilde{m_k}, R, \Delta)$. Let $\mathscr{B}(m_1,m_2,\cdots,m_k,\tilde{m_1},\tilde{m_2},\cdots,\tilde{m_k}, R, \Delta,\mu)$ be the number of solutions of $r,\tilde{r}$ to (\ref{mathscr B1}) which additionally satisfying $\gcd(r,\tilde{r})=\mu$.

We divide the solutions of $r$, $\tilde{r}$ into classes each one having fixed values $\gcd(r,\tilde{r})=\mu$, the points $\left(\frac{r}{\tilde{r}}\right)$ are spaced by $c(\beta)R^2\mu^{-2}$, where $c(\beta)$ is a constant. Then 
\begin{equation*}
  \mathscr{B}(m_1,m_2,\cdots,m_k,\tilde{m_1},\tilde{m_2},\cdots,\tilde{m_k}, R, \Delta,\mu)\ll1+\Delta(|\tilde{\alpha_k}|R^{\beta})^{-1}R^2\mu^{-2}
\end{equation*}

 Let $\mathscr{B}(M_1, M_2,\cdots,M_k, R, \Delta,\mu)$ be the number of solutions to (\ref{mathscr B}) which additionally satisfying $\gcd(r,\tilde{r})=\mu$. Summing over $m_1,m_2,\cdots,m_k,\tilde{m_1},\tilde{m_2},\cdots,\tilde{m_k}$ we have for $1\leqslant j\leqslant k$,
\begin{equation*}
\begin{aligned}
  \mathscr{B}(M_1, M_2,\cdots,M_k, R, \Delta,\mu)&\ll \sum_{\substack{m_j,\tilde{m_j}\sim M_j}}\mathscr{B}(m_1,m_2,\cdots,m_k,\tilde{m_1},\tilde{m_2},\cdots,\tilde{m_k}, R, \Delta,\mu)\\
  &\ll \sum_{\substack{m_j,\tilde{m_j}\sim M_j}}(1+\Delta(|\tilde{\alpha_k}|R^{\beta})^{-1}R^2\mu^{-2})\\
  &\ll(M_1M_2\cdots M_k)^2+\Delta M_1M_2\cdots M_k R^{2-\beta}\mu^{-2}\sum_{\substack{\tilde{m_j}\sim M_j}}(|\tilde{\alpha_k}|)^{-1}.
\end{aligned} 
\end{equation*}
Using Lemma \ref{sub alphak} we conclude that
\begin{equation*}
\begin{aligned}
  \mathscr{B}(M_1, M_2,\cdots,M_k, R, \Delta,\mu)\ll(M_1M_2\cdots M_k)^2+\Delta{M_1}^{(2^{k-2}+3/2)}(M_2\cdots M_k)^2R^{2-\beta}\mu^{-2}.
\end{aligned} 
\end{equation*}
Summing over $\mu$ we have
\begin{equation}\label{h-aspect}
\begin{aligned}
  \mathscr{B}(M_1, M_2,\cdots,M_k, R, \Delta)\ll R(M_1M_2\cdots M_k)^2+\Delta{M_1}^{(2^{k-2}+3/2)}(M_2\cdots M_k)^2R^{2-\beta}.
\end{aligned} 
\end{equation}
On the other hand, for fixed $r$, $\tilde{r}$, we denote the number of solutions of $m_1, \tilde{m_1}, m_2, \tilde{m_2},\cdots,m_k$, $\tilde{m_k}$ to
\begin{equation}\label{mathscr B2}
  \left|\left(\frac{r}{\tilde{r}}\right)^{\beta}(\alpha_k-\tilde{\alpha_k})\right|\leq\frac{\Delta}{R^{\beta}}
\end{equation}
by $\mathscr{B}(M_1, M_2,\cdots,M_k, r, \tilde{r}, R, \Delta)$. Then we can obtain that $\mathscr{B}(M_1, M_2,\cdots,M_k, R, \Delta)$ is\\ bounded by
\begin{equation*}
   \sum_{r,\tilde{r}\sim R}\mathscr{B}(M_1, M_2,\cdots,M_k, r, \tilde{r}, R, \Delta).
\end{equation*}

Applying Lemma \ref{(2.2) of Zhai} to the sequences $\{a_r\}=\{(r/\tilde{r})^{\beta}\alpha_k\}$ and $\{b_s\}=\{\tilde{\alpha_k}\}$, we get
\begin{equation*}
\begin{aligned}
  &\mathscr{B}(M_1, M_2,\cdots,M_k, r, \tilde{r}, R, \Delta)\\
  &\qquad\leqslant3\left(\#\left\{(m_1,\cdots,m_k,m_1',\cdots,m_k'):
  \left|\left(\frac{r}{\tilde{r}}\right)^{\beta}(\alpha_k-\alpha_k')\right|
  \leqslant\frac{\Delta}{R^{\beta}}\right\}\right)^{1/2}\\
  &\qquad\times\left(\#\left\{(\tilde{m_1},\cdots,\tilde{m_k},\tilde m_1',\cdots,\tilde m_k'):
  \left|\tilde{\alpha_k}-\tilde\alpha_k'\right|
  \leqslant\frac{\Delta}{R^{\beta}}\right\}\right)^{1/2}\\
  &\qquad\ll\Delta R^{-\beta}{M_1}^{-1/2}(M_1\cdots M_k)^2+{M_1}^{-1}(M_1\cdots M_k)^2,
\end{aligned}
\end{equation*}
where we use Lemma \ref{N O solutions}, and $\alpha_k'$, $\tilde\alpha_k'$ are defined similar to $\alpha_k$ and $\tilde{\alpha_k}$.

Summing over $r$ and $\tilde{r}$ we have
\begin{equation}\label{m-aspect}
\begin{aligned}
  \mathscr{B}(M_1, M_2,\cdots,M_k, R, \Delta)&\ll\sum_{r,\tilde{r}\sim R}\mathscr{B}(M_1, M_2,\cdots,M_k, r, \tilde{r}, R, \Delta)\\
  &\ll\Delta R^{2-\beta}{M_1}^{3/2}(M_2\cdots M_k)^2+R^2{M_1}(M_2\cdots M_k)^2.
\end{aligned}
\end{equation}
Above all, combining (\ref{h-aspect}) and (\ref{m-aspect}) we get
\begin{equation*}
\begin{aligned}
  &\mathscr{B}(M_1, M_2, R, \Delta)\\
  \ll&\min(R(M_1M_2\cdots M_k)^2+\Delta{M_1}^{(2^{k-2}+3/2)}(M_2\cdots M_k)^2R^{2-\beta},\\
  &\qquad\Delta R^{2-\beta}{M_1}^{3/2}(M_2\cdots M_k)^2+R^2{M_1}(M_2\cdots M_k)^2)\\
  \ll&\min(R(M_1M_2\cdots M_k)^2,R^2{M_1}(M_2\cdots M_k)^2)+\Delta R^{2-\beta}{M_1}^{(2^{k-2}+3/2)}(M_2\cdots M_k)^2\\
  \ll&R^{2-\beta}{M_1}^{3/2}(M_2\cdots M_k)^2+\Delta R^{2-\beta}{M_1}^{(2^{k-2}+3/2)}(M_2\cdots M_k)^2,
\end{aligned}
\end{equation*}
where we use $\min(a,b)\ll (ab)^{1/2}$.
\end{proof}

\section{Proof of the Theorem}

We are first going to estimate $\sum_{x\leqslant p\leqslant2x}\Delta^k(p)$. Using notations in Lemma \ref{delta 1 10 power moment} and Lemma \ref{delta 2 exp}, let $M$ be a real number such that $1\ll M\ll x$ and $k$ be an integer such that $3\leqslant k\leqslant9$, one has
\begin{equation}\label{delta^k exp}
\begin{aligned}
  \sum_{x<p\leqslant2x}\Delta^k(p)&=\sum_{x<p\leqslant2x}{\delta_1}^{k}(p,M)+O\left(\sum_{x<p\leqslant2x}{\delta_1}^{k-1}(p,M){\delta_2}(p,M)+\sum_{x<p\leqslant2x}{\delta_2}^{k}(p,M)\right)\\
  &:={\sum}_{11}+O\left({\sum}_{12}+{\sum}_{22}\right).
\end{aligned}
\end{equation}

For ${\sum}_{22}$, using (\ref{delta 2 A power moment}), for $M\ll x^{1/8}$ one has
\begin{equation}\label{sum 22 est}
  {\sum}_{22}\ll\left\{
  \begin{array}{lr}
    x^{1+\frac{k}{4}+\varepsilon}M^{-\frac{k}{4}},&2<k\leqslant4, \\
    x^{1+\frac{k}{4}+\varepsilon}M^{-\frac{A_0-k}{A_0-4}},& 4<k\leqslant9,
  \end{array}
  \right.
\end{equation}
where $A_0=262/27$.

For ${\sum}_{12}$, using Lemma \ref{delta 1 10 power moment} and (\ref{delta 2 A power moment}) and H\"{o}lder's inequality we obtain
\begin{equation}\label{sum 12 est}
  {\sum}_{12}\ll
  \left\{
  \begin{array}{lr}
    x^{1+\frac{k}{4}+\varepsilon}M^{-\frac{1}{4}},&2<k\leqslant8, \\
    x^{13/4+\varepsilon}M^{-\frac{A_0-9}{A_0-4}},& k=9.
  \end{array}
  \right.
\end{equation}

Next we turn to evaluate ${\sum}_{11}$. Using $\cos{\alpha}\cos{\beta}=\frac{1}{2}(\cos(\alpha-\beta)+\cos(\alpha+\beta))$ we have
\begin{equation}\label{sum 11 exp}
\begin{aligned}
  {\sum}_{11}&=\sum_{x<p\leqslant2x}\frac{p^{k/4}}{(\sqrt{2}\pi)^k}\sum_{m_1\leqslant M}\cdots\sum_{m_k\leqslant M}\frac{d(m_1)\cdots d(m_k)}{(m_1\cdots m_k)^{3/4}}\prod_{j=1}^{k}\cos\left(4\pi\sqrt{m_jp}-\frac{\pi}{4}\right)\\
  &=\frac{1}{(\sqrt{2}\pi)^{k}2^{k-1}}\sum_{\mathbf{i}_k\in\{0,1\}^{k-1}}\sum_{m_1\leqslant M}\cdots\sum_{m_k\leqslant M}\frac{d(m_1)\cdots d(m_k)}{(m_1\cdots m_k)^{3/4}}\sum_{x<p\leqslant2x}
  p^{k/4}\cos\left(4\pi\alpha_k\sqrt{p}-\frac{\pi}{4}\beta_k\right),
\end{aligned}
\end{equation}
where 
\begin{equation*}
\begin{aligned}
  \alpha_k&=\sqrt{m_1}+(-1)^{i_1}\sqrt{m_2}+\cdots+(-1)^{i_{k-1}}\sqrt{m_k},\\
  \beta_k&=1+(-1)^{i_1}+\cdots+(-1)^{i_{k-1}},\\
  \mathbf{i}_k&=(i_1,\cdots,i_{k-1}).
\end{aligned}
\end{equation*}
We divide the ${\sum_{11}}$ into two parts:
\begin{equation}\label{sum 11 chaifen}
  {\sum}_{11}:=\frac{1}{(\sqrt{2}\pi)^{k}2^{k-1}}\left({\sum}_{111}+{\sum}_{112}\right),
\end{equation}
where
\begin{equation*}
\begin{aligned}
  {\sum}_{111}&=\sum_{\mathbf{i}_k\in\{0,1\}^{k-1}}\cos(-\frac{\pi}{4}\beta_k)\mathop{\sum_{m_1\leqslant M}\cdots\sum_{m_k\leqslant M}}_{\substack{\alpha_k=0}}\frac{d(m_1)\cdots d(m_k)}{(m_1\cdots m_k)^{3/4}}\sum_{x<p\leqslant2x}p^{k/4},\\
  {\sum}_{112}&=\sum_{\mathbf{i}_k\in\{0,1\}^{k-1}}\mathop{\sum_{m_1\leqslant M}\cdots\sum_{m_k\leqslant M}}_{\substack{\alpha_k\neq0}}\frac{d(m_1)\cdots d(m_k)}{(m_1\cdots m_k)^{3/4}}
  \sum_{x<p\leqslant2x}p^{k/4}\cos\left(4\pi\alpha_k\sqrt{p}-\frac{\pi}{4}\beta_k\right).
\end{aligned}
\end{equation*}
First we consider the contribution of ${\sum}_{111}$. By the similar argument on (4.1)-(4.4) of Zhai\cite{MR2067871}, we obtain
\begin{equation}\label{sum 111 eva}
  {\sum}_{111}=B_k(\infty)\sum_{x<p\leqslant2x}p^{k/4}+O(x^{1+k/4+\varepsilon}M^{-1/2}),
\end{equation}
where
\begin{equation*}
  B_k(\infty)=\sum_{l=1}^{k-1}\binom{k-1}{l}s_{k,l}(\infty)\cos\frac{\pi(k-2l)}{4}.
\end{equation*}

Next we consider the contribution of ${\sum}_{112}$. One can see ${\sum}_{112}$ can be written as linear combination of $O(\log^kM)$ sums of the form
\begin{equation*}
\begin{aligned}
  &\sum_{\mathbf{i}_k\in\{0,1\}^{k-1}}\mathop{\sum_{m_1\sim M_1}\cdots\sum_{m_k\sim M_k}}_{\substack{\alpha_k\neq0}}\frac{d(m_1)\cdots d(m_k)}{(m_1\cdots m_k)^{3/4}}
  \sum_{x<p\leqslant2x}p^{k/4}\cos\left(4\pi\alpha_k\sqrt{p}-\frac{\pi}{4}\beta_k\right)\\
  =&\frac{1}{2}\sum_{\mathbf{i}_k\in\{0,1\}^{k-1}}\mathop{\sum_{m_1\sim M_1}\cdots\sum_{m_k\sim M_k}}_{\substack{\alpha_k\neq0}}\frac{d(m_1)\cdots d(m_k)}{(m_1\cdots m_k)^{3/4}}\\
  \times&\sum_{x<p\leqslant2x}p^{k/4}\left(\mathbf{e}\left(2\alpha_k\sqrt{p}-\frac{\beta_k}{8}\right)-\mathbf{e}\left(-2\alpha_k\sqrt{p}+\frac{\beta_k}{8}\right)\right),
\end{aligned}
\end{equation*}
where we use $\cos2\pi a=(\mathbf{e}(a)-\mathbf{e}(-a))/2$, and where $1\leqslant M_1,\cdots,M_k\leqslant M$, without loss of generality we assume that $\max(M_1,\cdots,M_k)=M_1$. By a splitting argument we only need to estimate
\begin{equation*}
  \mathcal{S}_1=\sum_{\mathbf{i}_k\in\{0,1\}^{k-1}}\mathop{\sum_{m_1\sim M_1}\cdots\sum_{m_k\sim M_k}}_{\substack{\alpha_k\neq0}}\frac{d(m_1)\cdots d(m_k)}{(m_1\cdots m_k)^{3/4}}\sum_{x<p\leqslant2x}p^{k/4}\mathbf{e}\left(2\alpha_k\sqrt{p}\right).
\end{equation*}
Trivially we have
\begin{equation}\label{S 1}
  \mathcal{S}_1\ll\frac{x^{k/4}}{(M_1\cdots M_k)^{3/4}}|{\mathcal{S}_1}^{*}|,
\end{equation}
where
\begin{equation*}
  {\mathcal{S}_1}^{*}=\sum_{\mathbf{i}_k\in\{0,1\}^{k-1}}\mathop{\sum_{m_1\sim M_1}\cdots\sum_{m_k\sim M_k}}_{\substack{\alpha_k\neq0}}d(m_1)\cdots d(m_k)
  \sum_{x<p\leqslant2x}\mathbf{e}\left(2\alpha_k\sqrt{p}\right).
\end{equation*}

It follows from partial summation that
\begin{equation}\label{S 1*}
\begin{aligned}
  {\mathcal{S}_1}^{*}&=\sum_{\mathbf{i}_k\in\{0,1\}^{k-1}}\mathop{\sum_{m_1\sim M_1}\cdots\sum_{m_k\sim M_k}}_{\substack{\alpha_k\neq0}}d(m_1)\cdots d(m_k)
  \int_{x}^{2x}\frac{d\mathcal{A}_k(u)}{\log u}\\
  &=\sum_{\mathbf{i}_k\in\{0,1\}^{k-1}}\mathop{\sum_{m_1\sim M_1}\cdots\sum_{m_k\sim M_k}}_{\substack{\alpha_k\neq0}}d(m_1)\cdots d(m_k)
  \left(\frac{\mathcal{A}_k(u)}{\log u}\bigg|_{x}^{2x}-\int_{x}^{2x}\frac{\mathcal{A}_k(u)}{u\log^2u}du\right),
\end{aligned}
\end{equation}
where
\begin{equation*}
  \mathcal{A}_k(u)=\sum_{p\leqslant u}\log p\:\mathbf{e}\left(2\alpha_k\sqrt{p}\right).
\end{equation*}
Moreover, we easily obtain
\begin{equation*}
  \mathcal{A}_k(u)=\sum_{n\leqslant u}\Lambda(n)\mathbf{e}\left(2\alpha_k\sqrt{n}\right)+O(u^{1/2}).
\end{equation*}

By a splitting argument we only need to give the upper bound estimate of the following sum
\begin{equation}\label{exp sum in this paper}
  \sum_{\mathbf{i}_k\in\{0,1\}^{k-1}}\mathop{\sum_{m_1\sim M_1}\cdots\sum_{m_k\sim M_k}}_{\substack{\alpha_k\neq0}}d(m_1)\cdots d(m_k)\sum_{n\sim x}
  \Lambda(n)\mathbf{e}\left(2\alpha_k\sqrt{n}\right).
\end{equation}

After using Heath-Brown's identity, i.e. Lemma \ref{H-B identity} with $k=3$, one can see that (\ref{exp sum in this paper}) can be written as linear combination of $O(\log^6x)$ sums, each of which is of the form
\begin{equation}\label{sum after H-B}
\begin{aligned}
  \mathcal{T}^{*}&:=\sum_{\mathbf{i}_k\in\{0,1\}^{k-1}}\mathop{\sum_{m_1\sim M_1}\cdots\sum_{m_k\sim M_k}}_{\substack{\alpha_k\neq0}}d(m_1)\cdots d(m_k)\\
  &\times{\sum_{n_1\sim N_1}}\cdots{\sum_{n_6\sim N_6}}(\log n_1)\mu(n_4)\mu(n_5)\mu(n_6)\mathbf{e}\left(2\alpha_k\sqrt{n_1\cdots n_6}\right),
\end{aligned}
\end{equation}
where $N_1\cdots N_6\asymp x$; $2N_i\leqslant(2x)^{1/3}$, $i=4,5,6$ and some $n_i$ may only take value $1$. Therefore, it is sufficient for us to estimate for each $\mathcal{T}^{*}$ defined as in
(\ref{sum after H-B}). Next, we will consider four cases.
\subsection*{Case 1}
\indent If there exists an $N_j$ such that $N_j\geqslant x^{3/4}$, then we must have $j\leqslant3$ for the fact that

\qquad\;\;\;$N_j\ll x^{1/3}$ with $j=4,5,6$. Let
\begin{equation*}
  r=\prod_{\substack{1\leqslant i\leqslant6\\i\neq j}}n_i,\qquad s=n_j,\qquad R=\prod_{\substack{1\leqslant i\leqslant6\\i\neq j}}N_i,\qquad S=N_j.
\end{equation*}
In this case, we can see that $\mathcal{T}^{*}$ is a sum of "Type I" satisfying $R\ll x^{1/4}$. By Lemma \ref{Type I est}, we have
\begin{equation*}
  x^{-\varepsilon}\cdot\mathcal{T}^{*}\ll x^{1/2}{M_1}^{5/4}(M_2\cdots M_k)+x^{3/4}{M_1}^{(2^{k-3}+3/4)}(M_2\cdots M_k).
\end{equation*}
\subsection*{Case 2}
\indent If there exists an $N_j$ such that $x^{1/2}\leqslant N_j< x^{3/4}$, then we take
\begin{equation*}
  r=\prod_{\substack{1\leqslant i\leqslant6\\i\neq j}}n_i,\qquad s=n_j,\qquad R=\prod_{\substack{1\leqslant i\leqslant6\\i\neq j}}N_i,\qquad S=N_j.
\end{equation*}
Thus, $\mathcal{T}^{*}$ is a sum of "Type II" satisfying $x^{1/4}\ll R\ll x^{1/2}$. By Lemma \ref{Type II est}, we have 
\begin{equation*}
\begin{aligned}
  x^{-\varepsilon}\cdot\mathcal{T}^{*}&\ll x^{7/8}(M_1\cdots M_k)+x^{15/16}{M_1}^{3/4}(M_2\cdots M_k)+x^{3/4}{M_1}^{(2^{k-3}+1)}(M_2\cdots M_k).
\end{aligned}
\end{equation*}
\subsection*{Case 3}
\indent If there exists an $N_j$ such that $x^{1/4}\leqslant N_j< x^{1/2}$, then we take
\begin{equation*}
  r=n_j,\qquad s=\prod_{\substack{1\leqslant i\leqslant6\\i\neq j}}n_i,\qquad R=N_j,\qquad S=\prod_{\substack{1\leqslant i\leqslant6\\i\neq j}}N_i.
\end{equation*}
Thus, $\mathcal{T}^{*}$ is a sum of "Type II" satisfying $x^{1/4}\ll R\ll x^{1/2}$. By Lemma \ref{Type II est}, we have 
\begin{equation*}
\begin{aligned}
  x^{-\varepsilon}\cdot\mathcal{T}^{*}&\ll x^{7/8}(M_1\cdots M_k)+x^{15/16}{M_1}^{3/4}(M_2\cdots M_k)+x^{3/4}{M_1}^{(2^{k-3}+1)}(M_2\cdots M_k).
\end{aligned}
\end{equation*}
\subsection*{Case 4} 
\indent If $N_j<x^{1/4}(j=1,2,3,4,5,6)$, without loss of generality, we assume that

\qquad\;\;$N_1\geqslant N_2\geqslant\cdots\geqslant N_6$. Let $\ell$ denote the natural number $j$ such that 
\begin{equation*}
  N_1N_2\cdots N_{j-1}<x^{1/4},\qquad N_1N_2\cdots N_j\geqslant x^{1/4}.
\end{equation*}
Since $N_1<x^{1/4}$ and $N_6<x^{1/4}$, then $2\leqslant\ell\leqslant5$. Thus we have
\begin{equation*}
  x^{1/4}\leqslant N_1N_2\cdots N_{\ell}=(N_1N_2\cdots N_{\ell-1})\cdot N_{\ell}<x^{1/4}\cdot x^{1/4}=x^{1/2}.
\end{equation*}
Let
\begin{equation*}
  r=\prod_{i=1}^{\ell}n_i,\qquad s=\prod_{i=\ell+1}^{6}n_i,\qquad R=\prod_{i=1}^{\ell}N_i,\qquad S=\prod_{i=\ell+1}^{6}N_i.
\end{equation*}
At this time, $\mathcal{T}^{*}$ is a sum of "Type II" satisfying $x^{1/4}\ll R\ll x^{1/2}$. By Lemma \ref{Type II est}, we have 
\begin{equation*}
\begin{aligned}
  x^{-\varepsilon}\cdot\mathcal{T}^{*}&\ll x^{7/8}(M_1\cdots M_k)+x^{15/16}{M_1}^{3/4}(M_2\cdots M_k)+x^{3/4}{M_1}^{(2^{k-3}+1)}(M_2\cdots M_k).
\end{aligned}
\end{equation*}
Combining the above four cases, we derive that
\begin{equation*}
\begin{aligned}
  x^{-\varepsilon}\cdot\mathcal{T}^{*}&\ll x^{7/8}(M_1\cdots M_k)+x^{15/16}{M_1}^{3/4}(M_2\cdots M_k)+x^{3/4}{M_1}^{(2^{k-3}+1)}(M_2\cdots M_k),
\end{aligned}
\end{equation*}
which combined with (\ref{S 1}) and (\ref{S 1*}) yields
\begin{equation}\label{sum 112 est}
\begin{aligned}
  &x^{-\varepsilon}\cdot{\sum}_{112}\\
  &\ll x^{7/8+k/4}(M_1\cdots M_k)^{1/4}+x^{15/16+k/4}(M_2\cdots M_k)^{1/4}+x^{3/4+k/4}{M_1}^{(2^{k-3}+1/4)}(M_2\cdots M_k)^{1/4}.
\end{aligned}
\end{equation}

Above all, since $M_1,\cdots,M_k\leqslant M$, by (\ref{sum 11 chaifen}), (\ref{sum 111 eva}), (\ref{sum 112 est}) we conclude that
\begin{equation}\label{sum 11 eva}
\begin{aligned}
  {\sum}_{11}&=\frac{B_k(\infty)}{(\sqrt{2}\pi)^k2^{k-1}}\sum_{x<p\leqslant2x}p^{k/4}\\
  &+O(x^{1+k/4+\varepsilon}M^{-1/2}+x^{7/8+k/4}M^{k/4}+x^{15/16+k/4}M^{(k-1)/4}+x^{3/4+k/4}{M}^{(2^{k-3}+k/4)},
\end{aligned}
\end{equation}
where $B_k(\infty)$ is defined in (\ref{sum 111 eva}).

Suppose $1\ll M\ll x^{1/8}$, combining (\ref{delta^k exp}), (\ref{sum 22 est}), (\ref{sum 12 est}) and (\ref{sum 11 eva}), taking 
\begin{equation*}
  M=\left\{
  \begin{array}{lr}
    x^{\min(\frac{1}{4k},\frac{1}{2^{k-1}+k+1})},&2<k\leqslant8, \\
    x^{\left(265+4\frac{A_0-k}{A_0-4}\right)^{-1}},& k=9,
  \end{array}
  \right.
\end{equation*}
 we conclude that
\begin{equation*}
  \sum_{x<p\leqslant2x}\Delta^k(p)=\frac{B_k(\infty)}{(\sqrt{2}\pi)^k2^{k-1}}\sum_{x<p\leqslant2x}p^{k/4}+O(x^{1+k/4-\theta(k,A_0)+\varepsilon}),
\end{equation*}
where
\begin{equation*}
   \theta(k,A_0)=\left\{
  \begin{array}{lr}
    {\min(\frac{1}{16k},\frac{1}{(2^{k+1}+4k+4)})},&2<k\leqslant8, \\
    {\left(4+265\frac{A_0-4}{A_0-k}\right)^{-1}},& k=9,
  \end{array}
  \right.
\end{equation*}
Thus replacing $x$ by $x/2$, $x/2^2$, and so on, and adding up all the results, we obtain
\begin{equation*}
  \sum_{p\leqslant x}\Delta^k(p)=\frac{B_k(\infty)}{(\sqrt{2}\pi)^k2^{k-1}}\sum_{p\leqslant x}p^{k/4}+O(x^{1+k/4-\theta(k,A_0)+\varepsilon}).
\end{equation*}
Thus we have completed the proof of the Theorem.

\section*{Acknowledgement}

The authors would like to appreciate the referee for his/her patience in refereeing this paper.
This work is supported by the Natural Science Foundation of Beijing Municipal (Grant No.1242003), and the National Natural Science Foundation of China (Grant No.11971476).

\bibliographystyle{plain}
\bibliography{reference.bib}

\begin{thebibliography}{10}

\bibitem{MR3231408}
Xiaodong Cao, Jun Furuya, Yoshio Tanigawa, and Wenguang Zhai.
\newblock On the differences between two kinds of mean value formulas of
  number-theoretic error terms.
\newblock {\em Int. J. Number Theory}, 10(5):1143--1170, 2014.

\bibitem{MR4421948}
Xiaodong Cao, Yoshio Tanigawa, and Wenguang Zhai.
\newblock On hybrid moments of {$\Delta_2(x)$} and {$\Delta_3(x)$}.
\newblock {\em Ramanujan J.}, 58(2):597--631, 2022.

\bibitem{MR1027058}
\'{E}tienne Fouvry and Henryk Iwaniec.
\newblock Exponential sums with monomials.
\newblock {\em J. Number Theory}, 33(3):311--333, 1989.

\bibitem{MR2176479}
Jun Furuya.
\newblock On the average orders of the error term in the {D}irichlet divisor
  problem.
\newblock {\em J. Number Theory}, 115(1):1--26, 2005.

\bibitem{MR1145488}
S.~W. Graham and G.~Kolesnik.
\newblock {\em van der {C}orput's method of exponential sums}, volume 126 of
  {\em London Mathematical Society Lecture Note Series}.
\newblock Cambridge University Press, Cambridge, 1991.

\bibitem{citekeyUnpublished}
Zhen Guo and Xin Li.
\newblock On a sum of the error term of the dirichlet divisor function over
  primes(submitted).
\newblock 2024.

\bibitem{MR1576556}
G.~H. Hardy.
\newblock The {A}verage {O}rder of the {A}rithmetical {F}unctions {P}(x) and
  delta(x).
\newblock {\em Proc. London Math. Soc. (2)}, 15:192--213, 1916.

\bibitem{MR0678676}
D.~R. Heath-Brown.
\newblock Prime numbers in short intervals and a generalized {V}aughan
  identity.
\newblock {\em Canadian J. Math.}, 34(6):1365--1377, 1982.

\bibitem{MR1295953}
D.~R. Heath-Brown and K.~Tsang.
\newblock Sign changes of {$E(T)$}, {$\Delta(x)$}, and {$P(x)$}.
\newblock {\em J. Number Theory}, 49(1):73--83, 1994.

\bibitem{MR1199067}
M.~N. Huxley.
\newblock Exponential sums and lattice points. {II}.
\newblock {\em Proc. London Math. Soc. (3)}, 66(2):279--301, 1993.

\bibitem{MR2005876}
M.~N. Huxley.
\newblock Exponential sums and lattice points. {III}.
\newblock {\em Proc. London Math. Soc. (3)}, 87(3):591--609, 2003.

\bibitem{MR0792089}
Aleksandar Ivi\'{c}.
\newblock {\em The {R}iemann zeta-function}.
\newblock A Wiley-Interscience Publication. John Wiley \& Sons, Inc., New York,
  1985.
\newblock The theory of the Riemann zeta-function with applications.

\bibitem{MR2342663}
Aleksandar Ivi\'c and Patrick Sargos.
\newblock On the higher moments of the error term in the divisor problem.
\newblock {\em Illinois J. Math.}, 51(2):353--377, 2007.

\bibitem{MR0883536}
Henryk Iwaniec and Andr\'{a}s S\'{a}rk\"{o}zy.
\newblock On a multiplicative hybrid problem.
\newblock {\em J. Number Theory}, 26(1):89--95, 1987.

\bibitem{MR0644943}
G.~Kolesnik.
\newblock On the order of {$\zeta ({\frac{1}{2}}+it)$} and {$\Delta (R)$}.
\newblock {\em Pacific J. Math.}, 98(1):107--122, 1982.

\bibitem{MR2475967}
Yuk-Kam Lau and Kai-Man Tsang.
\newblock On the mean square formula of the error term in the {D}irichlet
  divisor problem.
\newblock {\em Math. Proc. Cambridge Philos. Soc.}, 146(2):277--287, 2009.

\bibitem{MR3606942}
Jinjiang Li.
\newblock On the seventh power moment of {$\Delta(x)$}.
\newblock {\em Int. J. Number Theory}, 13(3):571--591, 2017.

\bibitem{MR1162488}
Kai~Man Tsang.
\newblock Higher-power moments of {$\Delta(x),\;E(t)$} and {$P(x)$}.
\newblock {\em Proc. London Math. Soc. (3)}, 65(1):65--84, 1992.

\bibitem{MR1512430}
J.~G. van~der Corput.
\newblock Zum {T}eilerproblem.
\newblock {\em Math. Ann.}, 98(1):697--716, 1928.

\bibitem{MR1580627}
Georges Voronoi.
\newblock Sur un probl\`eme du calcul des fonctions asymptotiques.
\newblock {\em J. Reine Angew. Math.}, 126:241--282, 1903.

\bibitem{MR1509041}
Georges Vorono\"{\i}.
\newblock Sur une fonction transcendante et ses applications \`a la sommation
  de quelques s\'{e}ries.
\newblock {\em Ann. Sci. \'{E}cole Norm. Sup. (3)}, 21:207--267, 1904.

\bibitem{MR2067871}
Wenguang Zhai.
\newblock On higher-power moments of {$\Delta(x)$}. {II}.
\newblock {\em Acta Arith.}, 114(1):35--54, 2004.

\bibitem{MR2168766}
Wenguang Zhai.
\newblock On higher-power moments of {$\Delta(x)$}. {III}.
\newblock {\em Acta Arith.}, 118(3):263--281, 2005.

\bibitem{MR2776009}
Deyu Zhang and Wenguang Zhai.
\newblock On the fifth-power moment of {$\Delta(x)$}.
\newblock {\em Int. J. Number Theory}, 7(1):71--86, 2011.

\end{thebibliography}

\end{document}